\documentclass[10pt]{amsart}

\usepackage{amsmath}
\usepackage{amsthm}
\usepackage{amsfonts}
\usepackage{amssymb,latexsym}
\usepackage[all]{xy}
\usepackage{graphicx}
\usepackage[mathscr]{eucal}
\usepackage{verbatim}
\usepackage{hyperref}

\theoremstyle{plain}
    \newtheorem{thm}{Theorem}[section]
    \newtheorem{prop}[thm]{Proposition}
    \newtheorem{lemma}[thm]{Lemma}
    \newtheorem{conj}[thm]{Conjecture}
    \newtheorem{cor}[thm]{Corollary}

\theoremstyle{definition}
    
    \newtheorem{defi}[thm]{Definition}
\theoremstyle{remark}
    \newtheorem{rem}[thm]{Remark}
    \newtheorem{example}[thm]{Example}

\numberwithin{equation}{section}

\newcommand{\rar}{\ensuremath{\rightarrow}}
\newcommand{\lrar}{\ensuremath{\longrightarrow}}

\newcommand{\la}{\langle}
\newcommand{\ra}{\rangle}
\newcommand{\Hom}{\textup{Hom}}

\newcommand{\stmod}{\textup{stmod}}

\newcommand{\uHom}{\underline{\Hom}}

\newcommand{\mcP}{\mathcal P}
\newcommand{\G}{\mathcal G}

\newcommand{\stk}[1]{\stackrel{#1}{\rightarrow}}
\newcommand{\lstk}[1]{\stackrel{#1}{\longrightarrow}}
\newcommand{\Ext}{\textup{Ext}}

\newcommand{\Dim}{\textup{Dim}}

\newcommand{\End}{\text{End}}

\def\HHH{\operatorname{H}\nolimits}
\def\HHHH{\operatorname{\hat{H}}\nolimits}
\def\Hom{\operatorname{Hom}\nolimits}
\def\Ext{\operatorname{Ext}\nolimits}
\def\hExt{\operatorname{\widehat{Ext}}\nolimits}
\def\Id{\operatorname{Id}\nolimits}
\def\res{\operatorname{res}\nolimits}

\def\CE{{\mathcal{E}}}
\def\CM{{\mathcal{M}}}

\def\CN{{\mathcal{N}}}
\def\CJ{{\mathcal{J}}}
\def\CK{{\mathcal{K}}}
\def\CX{{\mathcal{X}}}

\def\bZ{{\mathbb Z}}
\def\bF{{\mathbb F}}

\begin{document}

% Article information
\title{Finite generation of Tate cohomology}

\date{\today}

% author one information
\author[Jon F. Carlson]{Jon F. Carlson}
\address{Department of Mathematics \\
University of Georgia \\
Athens, GA 30602, USA}
\email{jfc@math.uga.edu}

% author two information
\author{Sunil K. Chebolu}
\address{Department of Mathematics \\
Illinois State University \\
Campus box 4520 \\
Normal, IL 61790, USA}
\email{schebol@ilstu.edu}
%\thanks{}

% author three information
\author{J\'{a}n Min\'{a}\v{c}}
\address{Department of Mathematics\\
University of Western Ontario\\
London, ON N6A 5B7, Canada}
\email{minac@uwo.ca}

\thanks{The first author is partially supported by a grant from  NSF and the
third  author is supported from NSERC}

% AMS information

% Abstract
\begin{abstract}

Let $G$ be a finite group and let $k$ be a field of characteristic $p$. Given 
a finitely generated indecomposable non-projective $kG$-module $M$, we
conjecture that if the Tate cohomology $\HHHH^*(G, M)$ of $G$ with coefficients
in $M$ is finitely generated over the Tate cohomology ring $\HHHH^*(G, k)$, 
then the support variety $V_G(M)$ of $M$ is equal to the entire maximal
ideal spectrum $V_G(k)$.  We prove various results  which support this
conjecture.  The converse of this conjecture is established for modules 
in the connected component of $k$ in the stable Auslander-Reiten quiver 
for $kG$, but it is shown to be false in general.
It is also shown that all finitely generated $kG$-modules over
a group $G$ have finitely generated Tate cohomology if and only if $G$ has
periodic cohomology.

\end{abstract}

\dedicatory{Dedicated to Professor Luchezar Avramov on his sixtieth birthday.}

\maketitle

\thispagestyle{empty}

%%\tableofcontents

%%%%%%%%%%%%%    section 1  %%%%%%%%%%%%%%%%%%%%
\section{Introduction}

Tate cohomology was introduced by Tate in his celebrated paper ~\cite{Tate}
where he proved the main theorem of class field theory in a remarkably simple
way using Tate cohomology. After Cartan and Eilenberg's treatment
~\cite{CarEil} of Tate cohomology and Swan's basic 
results on free group actions on
spheres ~\cite{Swan-PERIOD}, Tate cohomology became one of the basic tools in
current mathematics. Our aim in this paper is to address a fundamental question:
when is the Tate cohomology with coefficients in a module finitely generated 
over the Tate cohomology ring of the group. 

Suppose $G$ be a finite group and let $k$ be a field of characteristic $p$. If $M$
is a finitely generated $kG$-module, then a well-known result in group
cohomology due to Golod, Evens and Venkov says that $\HHH^*(G,M)$ is finitely
generated as a graded module over $\HHH^*(G, k)$. Our goal is to investigate a
similar finite-generation result for Tate cohomology. More precisely,  if $M$
is a finitely generated $kG$-module, then we want to know whether the
Tate cohomology $\HHHH^*(G, M)$ of $G$ with coefficients
in $M$ is  finitely generated as a graded module 
over the Tate cohomology ring $\HHHH^*(G, k)$. In Section \ref{Sec:motivation}
we explain one reason for being interested in this problem.
In general, it seems that the Tate cohomology of a module is seldom finitely
generated, which is a striking contrast to the situation with ordinary cohomology. 
However, there are some notable exceptions. 
Our investigations have led us to a conjecture which we state as follows.

\begin{conj} \label{conjec}
Let $G$ be a finite group and let $M$ be an indecomposable 
finitely generated $kG$-module such that $\HHH^*(G,M) \neq \{0\}$. 
If $\HHHH^*(G, M)$ is finitely generated
over $\HHHH^*(G, k)$, then  the support variety $V_G(M)$ of $M$ is equal to the
entire maximal ideal spectrum $V_G(k)$ of the group cohomology ring.
\end{conj}

The condition that $\HHH^*(G,M) \neq 0$ is certainly necessary since there are
many modules with proper support varieties and vanishing cohomology \cite{ben-car-rob}. 
Perhaps it is necessary to require that $M$ lies in the thick subcategory of the stable 
category generated by $k$.
 
We have evidence for the conjecture from two directions. First, the results of 
\cite{negtate} indicate that products in negative Tate cohomology are often
zero and we can use this to develop boundedness conditions on finitely 
generated modules over Tate cohomology. Under the right circumstances, 
these conditions imply infinite generation of the Tate cohomology. Secondly, 
for groups having $p$-rank at least two, we can show that many periodic
modules fail to have finitely generated Tate cohomology. Indeed, we prove that
for any such group there is at least one module whose Tate cohomology is not
finitely generated. Hence, the only groups having the property that every
finitely generated $kG$-module has finitely generated Tate cohomology have
$p$-rank one or zero. 

On the other hand, in general, there are numerous modules which have finitely
generated Tate cohomology. In the last section we show some ways in which 
these modules can be constructed. It turns out that the constructions are 
consistent with the Auslander-Reiten quiver for $kG$-modules. That is, if a 
nonprojective module in a connected 
component of the Auslander-Reiten quiver has finitely
generated Tate cohomology, then so does every module in that component. 

The paper is organized as follows. We begin in Section \ref{Sec:motivation} by
explaining how we had naturally arrived at the problem of finite generation of
Tate cohomology. Sections \ref{sec:moduleswithBFGS} and
\ref{sec:periodicmodules} deal with modules whose Tate cohomology is not
finitely generated and contain proofs in the two 
directions mentioned above. In Section \ref{sec:almostsplit} we 
prove  affirmative results which provides a good source of modules 
whose Tate cohomology is finitely generated.

Throughout the paper $G$ denotes a non-trivial finite 
group, and all $kG$-modules are
assumed to be finitely generated. We use standard facts and notation of
the stable module category  of $kG$ \cite{carlson-modulesandgroupalgebras},
support varieties \cite{Bbook, CTVZ}, and of almost split sequences 
\cite{Bbook}.

%%%%%%%%%%%%%  section 2  %%%%%%%%%%%%%%%%%%%%%%

\section{Universal ghosts in $\stmod(kG)$} \label{Sec:motivation}
Here we explain briefly how we had arrived at the problem of finite
generation of Tate cohomology. More details can be found in \cite{CCM2, CCM}.
The question of finite generation is very natural. The finite
generation of the ordinary cohomology has been very important in the
development of the theory of support varieties and in other connections. 
For Tate cohomology, almost nothing is known about the the question of finite
generation beyond what is in this paper. 

The following natural question was raised in \cite{CCM}:  when does the
Tate cohomology functor detect trivial maps in the stable module category
$\stmod(kG)$ of finitely generated $kG$-modules? A map 
$\phi\colon M \rar N$ between finitely generated
$kG$-modules is said to be a \emph{ghost} if the induced map in Tate
cohomology groups
\[ \uHom_{kG}(\Omega^i k, M) \lrar \uHom_{kG}(\Omega^i k, N)\]
is zero for each integer $i$. With this definition, the above question is
equivalent to asking when  every ghost map in $\stmod(kG)$ trivial.  
In addressing this question, it is convenient to have a universal
ghost out of any finitely generated $kG$-module $M$ in $\stmod(kG)$, i.e., a
ghost map $\phi \colon M \rar N$ in $\stmod(kG)$ such that  every ghost out of
$M$ factors through $N$ via $\phi$. The point is that if the 
universal ghost vanishes, then all ghosts vanish.

So the  problem  boils down to finding a universal ghost out of $M$ (if it
exists) for every module $M$ in $\stmod(kG)$. The point is that if the Tate cohomology
$\HHHH^*(G,M)$ is finitely generated as a graded module over $\HHHH^*(G, k)$,
then a universal ghost out of $M$ can be constructed in 
$\stmod(kG)$. This is done as follows. Let $\{ v_j \}$ be a finite set of
homogeneous generators for $\HHHH^*(G,M)$ as a $\HHHH^*(G, k)$-module. These
generators can be assembled into a map
\[
\bigoplus_j \, \Omega^{|v_j|}\,k \lrar M
\]
in $\stmod(kG)$. This map can then be completed to a triangle
\begin{equation} \label{eq:univ-ghost}
\bigoplus_j \, \Omega^{|v_j|}\,k \lrar M \lstk{\Psi_M} F_M .
\end{equation}
By construction, it is clear that the first map in the
above triangle is surjective on the functors $\uHom_{kG}(\Omega^l k, -)$
for each $l$.
Therefore, the second map $\Psi_M$ must be a ghost. Thus we have
the following proposition.

\begin{prop}\label{prop:fduniv}
Suppose that  $M$ is a finitely generated $kG$-module such that $\HHHH^*(G,M)$ is
finitely generated as a graded module over $\HHHH^*(G, k)$. Then the map
$\Psi_M \colon M \rar F_M$ in the above triangle is a universal ghost out of
$M$.
\end{prop}

\begin{proof}
Universality of $\Psi_M$ is easy to see.
For the last statement, we note that because the sum is finite, $\oplus_j
\Omega^{|v_j|}\,k$ is finitely generated. 
\end{proof}

%%%%%%%%%%%%%%%   section 3   %%%%%%%%%%%%%%%%%%%%

\section{Modules with bounds on finitely generated submodules} 
\label{sec:moduleswithBFGS}

In this section we apply our main method for showing that modules
over Tate cohomology are not finitely generated. We explore the 
implications of the following condition. 

The material in this section draws heavily on the methods introduced in the paper
\cite{negtate}.

\begin{defi} \label{BFGS}
We say that a graded module $T = \oplus_{n \in \bZ} T^n$ over $\HHHH^*(G,k)$
has bounded finitely generated submodules if for any $m$ there is a 
number $N = N(m)$ such that the submodule $S$
of $T$ generated by $\oplus_{n>m} T^n$ is contained in
$\oplus_{n>N} T^n$.  
\end{defi}

\begin{lemma} \label{bfgs=nfg}
If a graded module $T = \oplus_{n \in \bZ} T^n$ 
over $\HHHH^*(G,k)$ has bounded finitely generated submodules and
if $T^n \neq \{0\}$ for arbitrarily
small (meaning negative) values of $n$, 
then $T$ is not a finitely
generated module over $\HHHH^*(G,k)$.
\end{lemma}

\begin{proof}
The proof is an immediate consequence of the definition. The point is that
any finitely generated submodule of $T$ is contained in 
$\sum_{n>m} T^n$ for some $m$ and hence cannot generate all of $T$.
\end{proof}

\begin{rem} \label{submaxgrowth} 
There is a more general formulation of the boundedness condition that might
be useful, though we do not use it in this paper. 
We say that $T^*$ has submaximal growth of finitely 
generated submodules if the degree of the pole at 1 of the 
Poincar\'e series for the submodule $S$
of $T$ generated by $\oplus_{n>m} T^n$ is strictly smaller than the degree 
of the pole at 1 of the Poincar\'e series of $T$.
The Poincar\'e series for $T$ is the Laurent series
\[
f_T(t) = \sum_{n = -\infty}^{\infty} (\Dim(T^n))t^n.
\]
Its pole at $t = 1$ is a measure of the growth rate of $T$ in negative degrees. 
That is, if the pole has degree $d$, then there is a number $c$ such that 
$\Dim(T^{-n}) \leq cn^{d-1}$ for all $n$, while for any constant $c$ there exists
a natural number $n$ such that $\Dim(T^{-n}) > cn^{d-2}$.
It is straightforward to show that any $T^*$ which has submaximal growth 
of finitely generated submodule 
is not finitely generated over $\HHHH^*(G,k)$. 
\end{rem}

The graded modules over the Tate cohomology ring that we are interested in 
have the form $\HHHH^*(G, L)$, where $L$ is a $kG$-module. 
We remind the reader that if $\HHHH^i(G, L) \ne 0$ for some $i$, then it is 
also non-zero for infinitely many negative and infinitely many positive values of $i$
\cite[Thm. 1.1]{ben-car-rob}. 
Moreover, a standard argument shows that any non-projective module $L$
in the thick subcategory generated by $k$ has non-vanishing Tate cohomology.
Here, the thick
subcategory generated by $k$ is the smallest full subcategory of $\stmod(kG)$
that contains $k$ and is closed under exact triangles and direct summands.

We use the next lemma several times in what follows. 

\begin{lemma} \label{basicseq}
Suppose that we have an exact sequence
\[
\xymatrix{
\CE: \quad 0 \ar[r] & L \ar[r] & M \ar[r] & N \ar[r] & 0
}
\]
of $kG$-modules where $\CE$ represents an element $\zeta$ in
$\Ext1_{kG}(N,L)$. Cup product with the element $\zeta$ induces a
homomorphism $\zeta: \HHHH^*(G,N) \longrightarrow \HHHH^*(G,L)[1]$.
Let $\CK^{*}$ be the kernel of the multiplication by $\zeta$,
and let $\CJ^{*}$ be the cokernel of multiplication by $\zeta$.
Then we have an exact sequence of $\HHHH^*(G,k)$-modules
$$
\xymatrix{
0 \ar[r] & \CJ^* \ar[r] & \HHHH^*(G, M) \ar[r] &
\CK^* \ar[r] & 0.
}
$$
Moreover, if $\CK^*$ is not finitely generated over $\HHHH^*(G,k)$, then neither
is $\HHHH^*(G,M)$.
\end{lemma}

\begin{proof}
The proof is a straightforward consequence of the naturality
of the long exact
sequence on Tate cohomology. That is, we have a sequence
\[
\xymatrix{
\dots \ar[r]^{\zeta \qquad} & \HHHH^{n}(G,L) \ar[r] & \HHHH^n(G,M)
\ar[r] & \HHHH^n(G,N) \ar[r]^{\zeta} & \HHHH^{n+1}(G,L) \ar[r] &
\dots
}
\]
and we note that the collection of the maps $\zeta$ in the long
exact sequence is a map of degree $1$ of $\HHHH^*(G,k)$-modules
\[
\xymatrix{
\zeta: \quad \HHHH^*(G,N) \ar[r]& \HHHH^*(G,L)[1].
}
\]
(The symbol $\CX[i]$ indicates the shift of the
$\HHHH^*(G,k)$-module $\CX$ by $i$ degrees.)

The last statement is a consequence of the fact that quotient modules of finitely generated
modules are finitely generated. 
\end{proof}

Now suppose that $\zeta \in \HHH^d(G,k)$ for $d > 0$ and that
$\zeta \neq 0$. We have an exact sequence
$$
\xymatrix{
0 \ar[r] & L_{\zeta} \ar[r] & \Omega^d k  \ar[r]^{\zeta}
& k \ar[r] & 0
}
$$
where $\zeta$ in the sequence is a homomorphism (uniquely) representing
the cohomology element $\zeta$. In the corresponding long exact
sequence on Tate cohomology
$$
\xymatrix{
\dots \ar[r]^{\zeta \qquad} & \HHHH^{n-1}(G,k) \ar[r] & \HHHH^n(G,L_{\zeta})
\ar[r] & \HHHH^n(G,\Omega^d k ) \ar[r]^{\zeta} & \HHHH^n(G,k) \ar[r] &
\dots,
}
$$
the homomorphism labeled $\zeta$ is multiplication by $\zeta$. That is, it is
degree $d$ map:
$$
\zeta:\HHHH^*(G,k)[-d] \ \ \lrar \ \ \HHHH^{*}(G,k).
$$
Here we are using the fact
that $\HHHH^s(G,\Omega^d k) \cong \HHHH^{s-d}(G,k)$.

As a result, we have, as in  Lemma \ref{basicseq},
an exact sequence of $\HHHH^*(G,k)$-modules
$$
\xymatrix{
0 \ar[r] & \CJ^*[-1] \ar[r] & \HHHH^*(G, L_{\zeta}) \ar[r] &
\CK^*[-d] \ar[r] & 0.
}
$$
where $\CJ^*$ and $\CK^*$ are the cokernel and kernel of multiplication
by $\zeta$, respectively.

\begin{lemma} \label{KJbounds}
Suppose that $\zeta \in \HHH^d(G,k)$ is a regular element on
$\HHH^*(G,k)$. Then
\begin{enumerate}
\item
$\CK^m = \{0\}$ for all $m \geq 0$, and
\item
$\CJ^m = \{0\}$ for all $m < 0$.
\end{enumerate}
\end{lemma}

\begin{proof}
The first statement is the definition that $\zeta$ is a regular element
in $\HHH^*(G,k)$. The second statement is a consequence of Lemma 3.5
of \cite{negtate}. For the sake of completeness we include a proof.
 For $t > 0$,  let
$$
\xymatrix{
\langle \ \ , \ \ \rangle: \HHHH^{-t-1}(G,k) \otimes \HHHH^t(G,k) \quad
\ar[r] & \quad \HHHH^{-1}(G,k) \ \cong \ k
}
$$
be the Tate duality. Let $\zeta_1, \dots, \zeta_s$ be a $k$-basis
for $\HHHH^{-m - 1}(G,k)$. Then because multiplication by $\zeta$,
$$
\xymatrix{
\HHHH^{-m - 1}(G,k) \quad \ar[r] & \quad \HHHH^{-m + d - 1}(G,k)
}
$$
is a monomorphism (since $-m-1 \ge 0$), the elements $\zeta\zeta_1, \dots,
\zeta\zeta_s$ are linearly independent. So there must exist elements
$\gamma_1, \dots, \gamma_s$ in $\HHHH^{m-d}(G,k)$ such that
for all $i$ and $j$, we have
$$
\langle \gamma_i, \zeta\zeta_j \rangle \ = \ \langle \gamma_i\zeta, \zeta_j
\rangle \ = \ \delta_{i,j}
$$
where by $\delta_{i,j}$ we mean the usual Kronecker delta. A consequence
of this is that the elements $\gamma_1\zeta, \dots, \gamma_s\zeta$ must
be linearly independent and hence must form a basis for $\HHHH^m(G,k)$.
This proves the lemma.
\end{proof}

There are many examples of groups for which all products in negative 
cohomology are zero. For example we remind the reader of the following
theorem from \cite{negtate}.

\begin{thm} \label{neqcohoeqzero}
Suppose that the ordinary cohomology ring $\HHH^*(G,k)$ has a regular
sequence of length 2. Then the product of any two elements in negative
cohomology is zero. In particular, this happens whenever the $p$-rank
of the center of a Sylow $p$-subgroup of $G$ is at least 2. 
\end{thm}

The second statement of the theorem was proved by Duflot
(see Theorem 12.3.3 of \cite{CTVZ}). 

\begin{prop} \label{regelt1}
Suppose that $G$ has $p$-rank at least two and that 
$\HHHH^*(G,k)$ has the property that the product of any
two elements in negative degrees is zero. If $\zeta \in \HHH^d(G,k)$
($d > 0$) is a regular element for $\HHH^*(G,k)$, then
$\HHHH^*(G, L_{\zeta})$ is not
finitely generated as a module over $\HHHH^*(G, k)$.
\end{prop}

\begin{proof}
As before, let $\CK^*$ be the kernel of the multiplication by $\zeta$ on
$\HHHH^*(G,k)$. The fact that $\CK^*$ is not zero in infinitely many 
negative degrees follows easily from Lemma 2.1 of \cite{negtate} 
and the fact that there is no bound on the dimensions of the spaces
$\HHHH^n(G,k)$ for negative values of $n$.
We have shown that $\CK^*$ 
has elements only in negative degrees and products of elements in
negative degrees are zero. Therefore, $\CK$ has bounded finitely generated
submodules and by Lemma \ref{bfgs=nfg} it is not finitely generated. Then by 
Lemma \ref{basicseq} neither is $\HHHH^*(G, L_{\zeta})$.
\end{proof}

\begin{example} \label{example:Klein}
We consider the Klein four group $G = V_4$.  The
classification of the indecomposable $kV_4$-modules
over a field $k$ of characteristic $2$ is well-known;
see \cite[Vol. 1, Thm. 4.3.2]{Bbook} for instance.
If the field $k$ is algebraically closed then every even dimensional
indecomposable non-projective module has the form $L_{\zeta^m}$ for 
some $\zeta \in \HHH1(H,k)$. On the other hand, every 
indecomposable module of odd dimension is isomorphic to 
$\Omega^i\, k$ for some $i$.  Because every nonzero element
of $\HHH1(G,k)$ is regular, we have that for any indecomposable
$kG$-module $M$, the Tate cohomology of $M$ is finitely generated
over $\HHHH^*(G,k)$ if and only if $V_G(M) = V_G(k)$.
In particular, Conjecture \ref{conjec} holds in this case. 
\end{example}

At this point we need to recall a technical notion. We say that a 
cohomology element $\zeta \in \HHH^n(G,k)$ 
annihilates the cohomology of a module 
$M$, if the cup product with $\zeta$ is the zero operator on 
$\Ext^*_{kG}(N,M)$ for all modules $N$. The element $\zeta$ annihilates 
the cohomology of $M$ if and only if $L_{\zeta} \otimes M \cong 
\Omega^n M \oplus \Omega M \oplus P$ where $P$ is some projective 
module. See Section 9.7 of \cite{CTVZ}. From the same source we have 
that if $p>2$ and if $\zeta \in \HHH^*(G,k)$ with $n$ even, then 
$\zeta$ annihilates the cohomology of $L_{\zeta}$.

Even in the case that $p=2$, we know that the degree one elements 
corresponding to maximal subgroups of a 2-group have the property that
$\zeta$ annihilates the cohomology of $L_{\zeta}$. Moreover, the product
of any two elements with this property has this property. 

We are now prepared to prove the main theorem of this section.

\begin{thm} \label{theorem-regular}
Suppose that $\HHHH^*(G,k)$ has the property that the product of any two
elements in negative degrees is zero. Let $\zeta \in \HHH^*(G,k)$ be a
regular element of degree $d$.  In the case that $p=2$, assume that 
$\zeta$ annihilates the cohomology of $L_{\zeta}$.
If $M$ is a finitely generated
$kG$-module such that $\HHHH^*(G, M) \ne 0$ 
and $V_G(M) \subseteq V_G \la \zeta \ra$, 
then $\HHHH^*(G,M)$ is not finitely generated as an $\HHHH^*(G,k)$-module.
\end{thm}

\begin{proof}
Since $\HHHH^*(G, M) \ne 0$, by Lemma \ref{bfgs=nfg} 
it is enough to show that 
$\HHHH^*(G, M)$ has bounded finitely generated submodules. 
Because of the condition that 
$V_G(M) \subseteq V_G \la \zeta \ra$, we know 
that some power of $\zeta$, say $\zeta^t$, 
annihilates the cohomology of $M$. 
Hence it follows that
$$
L_{\zeta^t} \otimes M \ \cong \ \Omega M \oplus \Omega^{td} M \oplus P,
$$
for some projective module $P$. Thus, $\HHHH^*(G,M)$ 
has bounded finitely generated submodules 
if and only if $\HHHH^*(G,L_{\zeta^t} \otimes M)$
also has this property. Note that if $p >2$, then the degree of $\zeta$
must be even because $\zeta$ is regular and hence not nilpotent. 
So for any value of $p$ we have that $\zeta$ annihilates the 
cohomology of $L_{\zeta}$.

The action of $\HHHH^*(G,k)$
on $\HHHH^*(G, L_{\zeta^t} \otimes M)$ factors
through the map $\HHHH^*(G,k) \longrightarrow \hExt_{kG}^*(L_{\zeta^t},
L_{\zeta^t}) \cong \HHH^*(G, \ (L_{\zeta^t})^* \otimes L_{\zeta^t}) \cong 
\HHH^*(G, \ \Omega^{-dt}L_{\zeta^t} \oplus \Omega^{-1}L_{\zeta^t})$, 
since for any $\zeta$ of degree $d$ we have that $L_{\zeta}^* \cong 
\Omega^{-d-1} L_{\zeta}$ (see \cite{CTVZ}, Section 11.3). 
So the target of that map has
bounded finitely generated submodules.

Now let $m$ be any integer. Without loss of generality we can assume
that $m < 0$. Let
$$
\CM = \bigoplus_{n \geq m} \HHHH^n(G,L_{\zeta^t} \otimes M)
\subseteq \left(\bigoplus_{n \geq m} 
\hExt_{kG}^n(L_{\zeta^t},L_{\zeta^t})\right)
\left( \bigoplus_{n \geq m} \HHHH^n(G, L_{\zeta^t} \otimes M)\right).
$$
{}From Definition \ref{BFGS}, we know that there exists a number $N$ such
that
$$
\HHHH^*(G,k) \cdot \bigoplus_{n \geq m} \hExt_{kG}^n(L_{\zeta^t},L_{\zeta^t})
\subseteq \bigoplus_{n \geq N} \hExt_{kG}^n(L_{\zeta^t},L_{\zeta^t}).
$$
Hence, we have that
\begin{align*}
\HHHH^*(G,k) \cdot \CM & \subseteq \HHHH^*(G,k) \cdot
\left(\bigoplus_{n \geq m} \hExt_{kG}^n(L_{\zeta^t},L_{\zeta^t})\right)
\left(\bigoplus_{n \geq m} \HHHH^n(G, L_{\zeta^t} \otimes M)\right) \\
& \subseteq \left(\bigoplus_{n \geq N} 
\hExt_{kG}^n(L_{\zeta^t},L_{\zeta^t})\right)
\left( \bigoplus_{n \geq m} \HHHH^n(G, L_{\zeta^t} \otimes M)\right) \\
& \subseteq \bigoplus_{n \geq N+m} \HHHH^n(G, L_{\zeta^t} \otimes M).
\end{align*}
Therefore, $\HHHH^n(G, L_{\zeta^t} \otimes M)$ has bounded finitely generated 
submodules.
\end{proof}

Using the results of the theorem, we can settle Conjecture \ref{conjec}
in some special cases as in the following. 

\begin{cor}
Let $p > 2$. Suppose that the group $G$ has 
an abelian Sylow $p$-subgroup with $p$-rank at least two. If $M$ is a
finitely generated $kG$-module 
with $\HHH^*(G,M) \neq 0$ and if $V_G(M)$ is a proper
subvariety of $V_G(k)$, then $\HHHH^*(G,M)$ is
not finitely generated as a module over $\HHHH^*(G,k)$.
\end{cor}

\begin{proof}
If $V_G(M)$ is a proper subvariety of $V_G(k)$, then $V_G(M) \subseteq
V_G(\zeta)$ for some non-nilpotent element $\zeta \in \HHH^*(G,k)$. But
because the Sylow subgroup of $G$ is an abelian $p$-group, every non-nilpotent
element in $\HHH^*(G,k)$ is regular, and moreover, any two elements in negative
degrees in $\HHHH^*(G,k)$ have zero product.  So the proof is complete by the
previous theorem.
\end{proof}

\begin{rem}
We should note that $\HHHH^*(G,M)$ having infinitely generated Tate cohomology
does not require that it have bounded finitely generated submodules or even 
submaximal growth of finitely generated submodule (see \ref{submaxgrowth}).
For an example, consider the semidihedral 2-group $G$ of order 16 and let
$k = \bF_2$.
Let $M = L_{\zeta}$, where $\zeta \in \HHH1(G,\bF_2)$ is a nonnilpotent element. 
See the example in Section 4 of \cite{negtate}. Then it can be seen that 
$M \cong \Omega k_H^{\uparrow G}$ where $H$ is the subgroup defined by the 
class $\zeta$, that is, the maximal subgroup of $G$ on which $\zeta$ vanishes. 
So we see that $M \otimes M \cong \Omega M \oplus \Omega M \oplus (kG)^{12}$.
Hence, we can see by the results of the next section, that $\HHHH^*(G,M)$ is not
finitely generated as a module over $\HHHH^*(G,k)$. On the other hand, $\zeta$ is 
not a regular element, so the module $\CK^*$, which is the kernel of $\zeta$, 
does not have bounded
finitely generated submodules or submaximal growth of finitely generated submodules.
However, a careful analysis shows that $\CK^*$ is not finitely generated. 
\end{rem}

We end this section by showing that there is a counterexample to the converse of
our conjecture \ref{conjec}. We suspect that such examples are numerous. We give only an
outline of the proof in one example, leaving the details to the reader. 

\begin{prop}
There exists a module $M$ with $V_G(M) = V_G(k)$ such that $\HHHH^*(G,M)$
is not finitely generated over $\HHHH^*(G,k)$. 
\end{prop}

\begin{proof}[Sketch of Proof] 
Let $G = \langle
x, y \rangle$ be an elementary abelian group of order $p2$. Here 
$k$ has characteristic $p$. We assume that $p>2$. Let 
$H = \langle y \rangle$ and let $L = k_H^{\uparrow G}$ be the induced 
module. The module of our example is the extension $M$ in the non-split exact 
sequence
\[
\xymatrix{
\CE: \qquad 0 \ar[r] & k \ar[r]^\sigma & M \ar[r] & L \ar[r] & 0.
}
\]
The module $M$ can be described by generators
and relations as the quotient of $kG$ by the ideal generated by 
$(y-1)2$ and $(x-1)(y-1)$. The map $\sigma$ sends $1$ to $y-1$. 
Note that because the dimension of $M$ is relatively prime to $p$, 
we must have that $V_G(M) = V_G(k)$.

We have a sequence
$0 \rightarrow \CJ^* \rightarrow \HHHH^*(G, M) \rightarrow
\CK^* \rightarrow 0$ as in \ref{basicseq},
where $\CJ^*$ and $\CK^*$ are, respectively, the cokernel and kernel of 
the map $\HHHH^*(G, L) \lrar \HHHH^{*+1}(G, k)$ given by multiplying by
the class of $\CE$. 
Our interest is in the submodule $\CK^* \subseteq \HHHH^*(G,L)$. Because
$L = k_H^{\uparrow G}$, we have by the Eckmann-Shapiro Lemma that 
$\HHHH^*(G,L) \cong \HHHH^*(\langle y \rangle, k)$. Consequently, 
$\HHHH^*(G,L)$ has dimension one in every degree and the action of 
$\HHHH^*(G,k)$ on $\HHHH^*(G,L)$ factors through the restriction map
$\HHHH^*(G,k)  \lrar \HHHH^*(\langle y \rangle,k)$, 
which we know is the zero map in negative degrees.  Therefore $\CK^*$ 
has bounded finitely generated submodules. So by Lemmas \ref{bfgs=nfg}
and \ref{basicseq},
the proof is complete when we show 
that $\CK^*$ is not zero in  infinitely many negative degrees. 

We take the long exact sequence in cohomology
corresponding to the dual $\CE^*$ of the exact sequence $\CE$,
noting that the module $L$ is self dual. The connecting 
homomorphism is cup product with the class of the 
sequence $\CE^*$. By Eckmann-Shapiro, it is the restriction map followed 
by cup product with a nonzero class $\eta$
in $\HHH1(H,k)$. Since $\eta2 = 0$, we have that the 
image has dimension one, if $n$ is even and $\delta$ is 
the zero map if $n$ is odd. Hence because $\Dim \HHH^n(G,k) = n+1$,
we must also have that $\HHH^n(G,M^*)$ also has dimension $n+1$.

By Tate duality,  $\HHH^{-n}(G,M)$ is dual to 
$\HHH^{n-1}(G,M^*)$ for $n > 0$. Therefore $\HHH^{-n}(G,M)$ has dimension
$n$, which is the same as the dimension of $\HHH^{-n}(G,k)$. 
Returning to the long exact sequence corresponding to $\CE$, we argue by
dimensions that 
the connecting homomorphism is the zero map in every second degree. 
So we show that the dimension of $\CK^n$ is zero if $n$ is negative and even and
is one otherwise. This completes the proof.
\end{proof}

%%%%%%%%%%%%%%%   section 4   %%%%%%%%%%%%%%%%%%%%

\section{Periodic modules} \label{sec:periodicmodules}

In this section, we present our second piece of evidence for the Conjecture
\ref{conjec}.  We show that for any group $G$ with $p$-rank at least
2, there is a finitely generated module $M$ with the property that
$\HHHH^*(G,\End_k M)$ is not finitely generated as a $\HHHH^*(G,k)$-module.

We recall that a finite group $G$ has periodic cohomology, meaning that the 
trivial module $k$ is periodic, if and only if $G$ has $p$-rank zero
or one (see \cite{Bbook} or \cite{CTVZ}). 

\begin{thm} \label{prop:periodic}
Suppose that the group $G$ has $p$-rank at least 2. Let $M$ be a
non-projective periodic $kG$-module such that $\HHH^*(G,M) \neq 0$. 
Then $\HHHH^*(G,\Hom_k(M,M))
\cong \hExt_{kG}^*(M,M)$ is not finitely
generated as a $\HHHH^*(G,k)$-module.
Thus for any finite group $G$ such that $p$ divides the order of $G$, 
the Tate cohomology of every finitely generated
$kG$-modules is finitely generated over $\HHHH^*(G, k)$ if and only if 
$G$ has $p$-rank one, meaning that the
Sylow $p$-subgroup of $G$ is either a cyclic group or a generalized Quaternion
group.
\end{thm}

\begin{proof}
Let $E = \langle x_1, \dots, x_n \rangle$ be
a maximal elementary abelian $p$-subgroup
such that the restriction $M_E$ is not a free module. There exists
an element $\alpha = (\alpha_1, \dots, \alpha_n) \in k^n$ and a
corresponding cyclic shifted subgroup $\langle u_{\alpha} \rangle$,
$$
u_{\alpha} \ = \ 1 \ + \ \sum_{i=1}^n \alpha_i (x_i -1)
$$
such that the restriction of $M$ to $\langle u_{\alpha} \rangle$ is
not projective (see Section 5.8 of \cite{Bbook}). 
Hence, the identity homomorphism $\Id_M: M \longrightarrow M$ does not
factor through a projective $k\langle u_{\alpha} \rangle$-module.
As a consequence, the map $k \longrightarrow \Hom_k(M,M)$ 
which sends $1 \in k$ to $\Id_M$ must represent a non-zero class in
$\HHHH0(\langle u_{\alpha} \rangle, \Hom_k(M,M))$.

The next thing that we note is that the restriction map
$$
\res_{G, \langle u_{\alpha} \rangle}: \HHHH^d(G,k)
\longrightarrow \HHHH^d(\langle u_{\alpha} \rangle,k)
$$
is the zero map if $d < 0$. The reason is that the restriction map
$$
\res_{E, \langle u_{\alpha} \rangle}: \HHHH^d(E,k) \
\longrightarrow \HHHH^d(\langle u_{\alpha} \rangle,k)
$$
is zero by \cite{negtate} since $E$ has
rank at least 2. 

Now suppose that $M$ is periodic of period $t$. For every $m$ we
have that $\Omega^{mt} M \cong M$ and there exists an element
$$
\zeta_m \in \hExt^{mt}_{kG}(M,M) \ \cong \ \HHHH^{mt}(G, \Hom_k(M,M))
$$
such that $\zeta_m$ is not zero on restriction to $\langle u_{\alpha}
\rangle$. That is, $\zeta_m$ is represented by a cocycle
$$
k \quad \longrightarrow \quad \Hom_k(M,M) \cong \Omega^{mt}\Hom_k(M,M)
$$
which does not factor through a projective module on restriction
to $\langle u_{\alpha} \rangle$.

Suppose that $\HHHH^*(G,\Hom_k(M,M))$ is finitely generated
as a module over $\HHHH^*(G,k)$. Then there exist generators
$\mu_1, \dots, \mu_r$ of $\HHHH^*(G,\Hom_k(M,M))$,
having degrees $d_1, \dots, d_r$, respectively.
Choose an integer $m$ such that
$mt < \text{min}\{d_i\}$. We must have that
$\zeta_m = \sum_{i=1}^r \gamma_i \mu_i$
for some $\gamma_i \in \HHHH^{mt-d_i}(G,k)$. But now, for every $i$,
we have that $mt-d_i$ is negative. Hence $\res_{G, \langle u_{\alpha}
\rangle}(\gamma_i) = 0$ for every $i$. Therefore, since restriction
onto a shifted subgroup is a homomorphism we have that
$\res_{G, \langle u_{\alpha} \rangle}(\zeta_m) = 0$. But this is a
contradiction.

To prove the last statement of the theorem, we recall
that every finite group with non-trivial Sylow $p$-subgroup
admits a finitely generated non-projective
and periodic $kG$-module in the thick subcategory generated by $k$. 
If the group has $p$-rank one, then $k$ is such a module. If the
$p$-rank of $G$ is greater than one, then any tensor product
$L_{\zeta_1} \otimes \dots \otimes L_{\zeta_n}$ is periodic and is in
the thick subcategory generated by $k$, provided the dimension
of the variety $V_G(\zeta_1) \cap \dots \cap V_G(L_{\zeta_n})$ 
has dimension one (see Chapter 10 of \cite{CTVZ} or Chapter 5 of 
\cite{Bbook}, Volume 2). 
\end{proof}

There is one other concept which ties up well with finite generation of Tate
cohomology, and this is a ghost projective class in the $\stmod(kG)$. Consider
the pair  $(\mcP , \G )$, where $\mcP$ is a class of objects isomorphic in
$\stmod(kG)$ to finite direct sums of suspensions of $k$, and $\G$ is a class
of all ghosts in $\stmod(kG)$. Recall that a ghost is a map of $kG$-modules
that is zero in Tate cohomology. We say that $(\mcP , \G )$ is a ghost
projective class if the following 3 conditions are satisfied.
\begin{enumerate}
\item The class of all maps $X \rar  Y$ such that the 
composite $P \rar X \rar Y$ is zero for all $P$ in $\mcP$ and all maps
$P \rar X$ is precisely $\G$.
\item The class of all objects $P$ such that the composite 
$P \rar X \rar Y$ is zero for all maps
$X \rar Y$ in $\G$ and all maps
$P \rar X$ is precisely $\mcP$.
\item For each object $X$ there is an exact triangle 
$P \rar X \rar Y$ with  $P$ in $\mcP$ and $X \rar Y$ in $\G $.
\end{enumerate}

The first question that comes to mind is whether  the 
ghost projective class exists in $\stmod(kG)$.
We answer this in the next theorem.

\begin{thm} For $G$ a finite group, such that $p$ divides the order of $G$.
The ghost projective class exists in $\stmod(kG)$ if and only if $G$ has 
$p$-rank one.
\end{thm}

\begin{proof}
It is clear from the definition of 
a ghost that $\mcP$ and $\G$ are orthogonal, i.e., the
composite $P \rar M \stk{h} N$ is zero for all $P$ in $\mcP$,  for all
$h$ in $\G$, and all maps $P \rar M$. So by  \cite[Lemma 3.2]{ch} it
remains to show that for all finitely generated $kG$-modules 
$M$, there exists a triangle $P \rar
M \rar N$ such that $P$ is in $\mathcal{P}$ and $M \rar N$ is in
$\mathcal{G}$.  The exact triangle \eqref{eq:univ-ghost} has this form
in the case that the Tate cohomology of $M$ is finitely generated over
$\HHHH^*(G,k)$.

For the converse, suppose that $M$ is a finitely 
generated $kG$-module. Since the ghost projective class
exists, we have an exact triangle
\[ \bigoplus \Omega^i k \lstk{\oplus\theta_i} M \lstk{\rho} N \]
in $\stmod(kG)$ where $\rho$ is a ghost.  We 
claim that the finite set $\{\theta_i \}$ generate
$\HHHH^*(G, M)$ as a module over $\HHHH^*(G, k)$. 
To see this, consider any element $\gamma$ in
$\HHHH^t(G, M)$ represented by a cocycle 
$\gamma \colon \Omega^t k \rar M$. Since $\rho$ is a ghost,
we get the following commutative diagram:
\[
\xymatrix{
\bigoplus \Omega^i k \ar[r]^{\oplus \theta_i} & M \ar[r]^{\rho} & N \\
  & \Omega^t k \ar@{.>}[ul]^{\oplus r_i}  
\ar[u]^{\gamma} \ar[ur]_{\rho \gamma = 0}&
}
\]
{}From this diagram, we infer that $\gamma = \sum r_i \theta_i$. 
This shows that $\HHHH^*(G, M)$
is finitely generated, as desired. Hence, by Theorem \ref{prop:periodic},
the $p$-rank of $G$ is one. 
\end{proof}

%%%%%%%%%%%%%%%    section 5   %%%%%%%%%%%%%%%%%%%%

\section{Modules with finitely generated Tate cohomology}
\label{sec:almostsplit}

It is clear that any module $M$ which is a direct sum of Heller
translates $\Omega^n k$ has finitely generated Tate cohomology.
This is simply because $\HHHH^*(G,M)$ is a direct sum of copies
of $\HHHH^*(G,k)$ which have been suitably translated in degrees.
Also any finitely generated modules over a group with periodic
cohomology has finitely generated Tate cohomology.
In this section we show that in general there are many more
modules with this property. Notice that every one of the modules
which we discuss has the property that $V_G(M) = V_G(k)$,
consistent with Conjecture \ref{conjec}.

We first consider the Tate cohomology of modules
$M$ which can occur as the middle term of an exact sequence of the form
$$
\xymatrix{
0 \ar[r] & \Omega^m k  \ar[r] & M \ar[r] & \Omega^n k  \ar[r] & 0
}
$$
for some values of $m$ and $n$. Such a sequence represents
an element $\zeta$ in
$$
\Ext1_{kG}(\Omega^n k , \Omega^m k) \cong \hExt_{kG}^{n+1-m}(k,k)
\cong \HHHH^{n+1-m}(G,k).
$$
For the purposes of examining the Tate cohomology of $M$ there is no
loss of generality in applying the shift operator $\Omega^{-m}$.
Consequently we can assume that the sequence has the form
\begin{equation} \label{seq3}
\xymatrix{
0 \ar[r] & k \ar[r] & M \ar[r] & \Omega^n k \ar[r] & 0
}
\end{equation}
for some $n$ and that $\zeta \in \HHHH^{n+1}(G,k)$.

The principal result of this section is the following.

\begin{thm} \label{trivial-extensions}
Suppose that the cohomology of $G$ is not periodic and that for
the module $M$ and cohomology element $\zeta$ as above, the map
$$
\xymatrix{
\zeta: \HHHH^*(G,k) \ar[r] & \HHHH^*(G,k)
}
$$
given by multiplication by $\zeta$ has a finite dimensional image.
Then the Tate cohomology $\HHHH^*(G,M)$ is finitely generated
as a module over $\HHHH^*(G,k)$.
\end{thm}

There are many example of sequences satisfying the conditions of the
theorem. In particular, it is often the case that multiplication by an element 
$\zeta$ in negative cohomology will have finite dimensional image. 
An example is the element in degree $-1$ which represents the almost
split sequence for the module $k$. Details of this example are given below.
In addition, if the depth of $\HHH^*(G,k)$ is
two or more then all products involving
elements in negative degrees are zero, and the principal
ideal generated by any element in negative cohomology contains
no non-zero elements in positive degrees (see \cite{negtate}).
Hence, multiplication by any element $\zeta$ in negative cohomology
has finite dimensional image.

\begin{proof}
As in Lemma \ref{basicseq}, we have an exact sequence of $\HHHH^*(G,k)$-modules
$$
\xymatrix{
0 \ar[r] & \CJ^* \ar[r] & \HHHH^*(G, M) \ar[r] &
\CK^*[-n] \ar[r] & 0,
}
$$
where  $\CK^*$ is the kernel of
multiplication by $\zeta$ on
$\HHHH^{*}(G,k)$ and $\CJ^{*}$ is the cokernel.
By assumption, the image of multiplication by $\zeta$ has finite
total dimension. This means that in all but a finite number of
degrees $r$, multiplication by $\zeta$ 
is the zero map. Clearly, $\CJ^*$ is finitely generated over
$\HHHH^*(G,k)$. So, $\HHHH^*(G,M)$ is finitely
generated as a module over $\HHHH(G,k)$ if and only if $\CK^*$ has the same
property.

First we view $\CK^*$ as a module over the ordinary cohomology ring
$\HHH^*(G,k)$. The elements in non-negative degrees form a submodule $\CM^*
=\sum_{i \geq 0} \CK^i$, which is finitely generated over $\HHH^*(G,k)$.
Let $\CN^*$ be the
$\HHHH^*(G,k)$-submodule module of $\CK^*$ generated by $\CM^*$. Our 
objective is to show that $\CN^* = \CK^*$, thereby proving the finite generation 
of $\CK^*$.
We notice first that $\CK^n \subseteq \CN^*$ for $n\geq 0$. It remains only to
show the same for $n < 0$.

Because the quotient of $\HHHH^*(G,k)$ by $\CK^*$ is finite dimensional, 
we must have that $\HHHH^n(G,k) = \CK^n$ for $n$ sufficiently large. For some
sufficiently large $n$, we can find an element $\gamma$ in $\CK^n$
which is a regular element for the ordinary cohomology ring $\HHH^*(G,k)$.
For example, by Duflot's Theorem (see \ref{neqcohoeqzero}),  
any element whose restriction
to the center of a Sylow
$p$-subgroup of $G$ is not nilpotent will serve this purpose (see
\cite{Bbook} or \cite{CTVZ}). Let $\theta$ be the image of $\gamma$
in $\CK^n$. We know that $\theta$ is not zero. We also know that
multiplication by $\gamma$ is a surjective map
$$
\xymatrix{
\gamma: \HHHH^{m-n}(G,k) \ar[r] & \HHHH^m(G,k)
}
$$
whenever $m <0$ (see Lemma 3.5 of \cite{negtate}). 
Hence, for any $m < 0$, we must have that
$\HHHH^{m-n}(G,k)\theta = \CK^m$. Since $\theta \in \CN^*$, we
get that $\CK^m \subseteq \CN^*$ for all $m < 0$. Hence,
$\CK^* = \CN^*$ is finitely generated as a module over
$\HHHH^*(G,k)$.
\end{proof}

One application of the theorem is the following. 

\begin{cor} \label{cor:middleterm}
The middle term of the almost split sequence
\[
\xymatrix{
0 \ar[r]  & \Omega^{2} k \ar[r]^{\sigma} & M \ar[r] & k \ar[r] & 0 
}
\]
ending with $k$ has finitely generated Tate cohomology.
\end{cor}

\begin{proof}
If $G$ has $p$-rank zero or one, then by Theorem \ref{prop:periodic}, all 
modules have finitely generated Tate cohomology.
So we assume that  $G$ has $p$-rank at least 2.
The almost split sequence corresponds to an
element $\zeta$ in $\HHHH^{-1}(G, k)$.
One of the defining property
of the almost split sequence is that for any module $N$, the connecting
homomorphism $\delta$ in the corresponding sequence
\[
\xymatrix{
\dots  \ar[r] & \Hom_{kG}(N,M) \ar[r]^{\sigma^*} & \Hom_{kG}(N,k) \ar[r]^{\delta \quad} & 
\Ext1_{kG}(N, \Omega^{2} k) \ar[r] & \dots
}
\]
is non-zero if and only if $N \cong k$.  This connecting homomorphism is multiplication by 
$\zeta$. Now any element $\gamma$ in $\HHHH^d(G,k)$ is represented by 
a map $\gamma: \Omega^{-d} k \rightarrow k$. Hence, we see that 
$\zeta\gamma = 0$ whenever $d \neq 0$.  
Therefore, multiplication by $\zeta$ on $\HHHH^*(G, k)$
has finite-dimensional image. 
\end{proof}

\begin{prop}
Let $N$ be a finitely generated indecomposable non-projective $kG$-module
that is not isomorphic to $\Omega^i k$ for any $i$. Consider 
 the almost split sequence
\[ 
\xymatrix{
0 \ar[r] &  \Omega2 N \ar[r] &  M \ar[r] & N \ar[r] & 0
}
\]
ending in $N$. If $N$ has finitely generated Tate cohomology, then so does 
the middle term $M$.
\end{prop}

\begin{proof}
In a similar way as in the previous proof, for any $i$, 
the connecting homomorphism $\delta$
in the sequence 
\[
\xymatrix{
\dots  \ar[r] & \Hom_{kG}(\Omega^i k, M) \ar[r] & 
\Hom_{kG}(\Omega^i k, N) \ar[r]^{\delta \quad} & 
\Ext1_{kG}(\Omega^i k, \Omega^{2} N) \ar[r] & \dots
}
\]
is zero because $\Omega^i k \not\cong N$. 
As a consequence, $\delta$ induces the zero map on Tate cohomology. 
Hence, the long exact sequence in Tate cohomology breaks 
into short exact sequences:
\[ 
\xymatrix{
 0 \ar[r] & \HHHH^*(G, \Omega2 N) \ar[r] & \HHHH^*(G,  M) 
\ar[r] &  \HHHH^*(G, N) \ar[r] & 0
}
\]
It is now clear that if $N$ has finitely
generated Tate cohomology, then so does $M$.
\end{proof}

In summary, combining the last two results we have the following theorem.

\begin{thm}
Let $C$ be a connected component of the stable Auslander-Reiten quiver associated to
$kG$. Then either all modules in $C$ have finitely generated Tate cohomology 
or no module
in $C$ has this property. Moreover, all modules in the 
connected component of the quiver which contains
$k$ have finitely generated Tate cohomology.
\end{thm}

It is shown in ~\cite[Proposition 5.2]{Aus-Car} that 
all modules $M$ in the connected component of the quiver which contains
$k$ have the property $V_G(M) = V_G(k)$.  Thus the 
last theorem is consistent with Conjecture \ref{conjec}

%%%%%%%%%%%%%%%%   bottom matter  %%%%%%%%%%%%%%%%%%%%
\vskip.2in

\noindent
\textbf{Acknowledgments:}
The first author is grateful to 
the RWTH in Aachen for their hospitality and the Alexander von 
Humboldt Foundation for financial support during a visit to Aachen 
during which parts of this paper were written. 
The first and second  authors had some fruitful conversations on this 
work at MSRI during the spring 2008 semester on Representation theory and 
related topics. They both would like to thank MSRI for its hospitality.

\bibliographystyle{alpha}
%\bibliography{../lit}

\end{document}